\documentclass[reqno,a4paper]{amsart}
\usepackage{graphicx}
\usepackage{amsmath}
\usepackage{mathrsfs}
\usepackage{amsfonts}
\usepackage{amssymb}
\usepackage{enumerate}
\usepackage{latexsym}
\usepackage{graphpap}
\usepackage{tikz}

\newtheorem{theorem}{Theorem}[section]

\newtheorem{lemma}[theorem]{Lemma}

\theoremstyle{definition}

\newtheorem{algo}[]{Algorithm}

\theoremstyle{remark}

\numberwithin{equation}{section}

\newcommand{\bfa}{\mathbf{a}}
\newcommand{\bfb}{\mathbf{b}}

\newcommand{\bfp}{\mathbf{p}}
\newcommand{\bfq}{\mathbf{q}}

\newcommand{\bfs}{\mathbf{s}}
\newcommand{\bft}{\mathbf{t}}
\newcommand{\bfu}{\mathbf{u}}
\newcommand{\bfv}{\mathbf{v}}
\newcommand{\bfw}{\mathbf{w}}

\newcommand{\nfb}{non-finitely based}

\newcommand{\con}{\operatorname{\mathsf{con}}}

\newcommand{\occ}{\operatorname{\mathsf{occ}}}
\newcommand{\sylv}{\operatorname{\mathsf{sylv}}}
\newcommand{\baxt}{\operatorname{\mathsf{baxt}}}
\newcommand{\stal}{\operatorname{\mathsf{stal}}}
\newcommand{\taig}{\operatorname{\mathsf{taig}}}

\newcommand{\ev}{\operatorname{\mathsf{ev}}}
\newcommand{\M}{\operatorname{\mathsf{M}}}
\newcommand{\ip}{\operatorname{\mathsf{ip}}}
\newcommand{\fp}{\operatorname{\mathsf{fp}}}
\newcommand{\id}{\operatorname{\mathsf{id}}}
\newcommand{\mix}{\operatorname{\mathsf{mix}}}

\begin{document}

\title[Finite basis problems for plactic-like monoids]{Finite basis problems for stalactic, taiga, sylvester and  Baxter monoids}%
\thanks{} %

\author[B. B. Han]{Bin Bin Han }
\author[W. T. Zhang]{Wen Ting Zhang$^\star$}\thanks{$^\star$Corresponding author} %

\address{School of Mathematics and Statistics, Lanzhou University, Lanzhou, Gansu 730000, PR China; Key Laboratory of Applied Mathematics and Complex Systems, Lanzhou, Gansu 730000, PR China} %
\email{zhangwt@lzu.edu.cn}



\thanks{This research was partially supported by the National Natural Science Foundation of China (nos.~11771191, 11371177) and
the Natural Science Foundation of Gansu Province (no. 20JR5RA275).}

\begin{abstract}
Stalactic, taiga, sylvester and Baxter monoids arise from the combinatorics of tableaux by identifying words over a fixed ordered alphabet whenever they produce the same tableau via some insertion algorithm. In this paper, three sufficient conditions under which semigroups are finitely based are given. By applying these sufficient conditions, it is shown that all stalactic and taiga monoids of rank greater than or equal to $2$ are finitely based and satisfy the same identities, that all sylvester monoids of rank greater than or equal to $2$ are finitely based and satisfy the same identities and that all Baxter monoids of rank greater than or equal to $2$ are finitely based and satisfy the same identities.
\end{abstract}

\maketitle

\section{Introduction}
Knuth introduced the tableaux algebra \cite{K71} in the 1970s and this algebra was later studied in detail by Lascoux and Sch\"{u}tzenberger under the name plactic monoid \cite{LS81}. Plactic monoid arise from the combinatorics of tableaux by identifying words over a fixed ordered alphabet whenever they produce the same tableau via Schensted's insertion algorithm \cite{Sch61}. Plactic-like monoids which arise from the combinatorics of tableaux as the plactic monoid include the hypoplactic monoid \cite{KT97,Nov00}, the stalactic monoid \cite{HNT07,Pri13}, the taiga monoid \cite{Pri13}, the sylvester monoid \cite{HNT05} and the Baxter monoid \cite{Gir12}. These monoids have attracted much attention due
to their interesting connection with combinatorics \cite{Lot02} and applications in symmetric functions \cite{Mac08}, representation theory \cite{Ful97}, Kostka-Foulkes polynomials \cite{LS78,LS81}, Schubert polynomials \cite{LS85, LS90}, and musical theory \cite{Jed11}.

Each of these plactic-like monoids can be obtained by factoring the free monoid $\mathcal{A}^*$ over the infinite ordered alphabet $\mathcal{A} = \{1 < 2 < 3 < \cdots\}$ by a congruence that can be defined by a so-called insertion
algorithm that computes a combinatorial object from a word. For example, for
the stalatic monoid, the corresponding combinatorial objects are stalactic tableaus.
We introduce the definitions of combinatorial objects and insertion algorithms used to construct the stalactic, taiga, sylvester and Baxter monoids.

A \textit{stalactic tableau} is a finite array of symbols of $\mathcal{A}$ in which columns are top-aligned, and two symbols appear in the same column if and only if they are equal.
The associated insertion algorithm is as follows:

\begin{algo}\label{algo:stal}\cite[\S~3.7]{HNT07}
Input: A stalactic tableau $T$ and a symbol $a \in \mathcal{A}$.
If $a$ does not appear in $T$, add $a$ to the left of the top row of $T$; if $a$ does appear in $T$, add $a$ to the bottom of the column in which $a$ appears. Output the new tableau.
\end{algo}

Let $w_1, \cdots, w_k\in \mathcal{A}$ and $w=w_1\cdots w_k \in \mathcal{A}^*$. Then the combinatorial object ${\rm P}_{\stal_{\infty}}(w)$ of $w$ is obtained as follows: reading
$w$ from right-to-left, one starts with an empty tableau and inserts each symbol in $w$ into a stalactic tableau according to  Algorithm \ref{algo:stal}.
For example, ${\rm P}_{\stal_{\infty}}(3613151265)$ is given as follows:
\begin{equation*}
\begin{tikzpicture}
\matrix [nodes=draw,column sep=1mm]
{
\node {3}; &[-1mm] \node{1}; &[-1mm] \node{2};&[-1mm] \node{6}; &[-1mm] \node{5}; \\
\node {3}; & \node{1}; &  & \node{6};  & \node{5};\\
 & \node{1}; & &  &\\
};
\end{tikzpicture}
\end{equation*}
Notice that the order in which the symbols appear along the first row in ${\rm P}_{\stal_{\infty}}(w)$ is the same as the order
of the rightmost instances of the symbols that appear in $w$.

A \textit{binary search tree with multiplicities} is a labelled binary search tree in which each
label appears at most once, where the label of each node is greater than the label of every
node in its left subtree, and less than the label of every node in its right subtree, and where a non-negative integer called the \textit{multiplicity} is assigned to each node label.
The associated insertion algorithm is as follows:
\begin{algo}\label{algo:taig}\cite[ Algorithm 3]{Pri13}
Input: A binary search tree with multiplicities $T$ and a symbol $a \in \mathcal{A}$.
If $T$ is empty, create a node, label it by $a$, and assign it multiplicity $1$. If $T$ is non-empty, examine the label $x$ of the root node: if $a < x$, recursively insert $a$ into the left subtree of the root node; if $a > x$, recursively insert $a$ into the right subtree of the root note; and if $a = x$, increment by $1$ the multiplicity of the node label $x$.
\end{algo}

Let $w_1, \cdots, w_k\in \mathcal{A}$ and $w=w_1\cdots w_k \in \mathcal{A}^*$. Then the combinatorial object ${\rm P}_{\taig_{\infty}}(w)$ of $w$ is obtained as follows: reading
$w$ from right-to-left, one starts with an empty tree and inserts each symbol in $w$ into a binary search tree with multiplicities according to  Algorithm \ref{algo:taig}.
For example, ${\rm P}_{\taig_{\infty}}(3613151265)$ is given as follows:
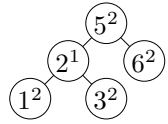
\begin{equation*}
\begin{tikzpicture}
[grow'=down,line width = 0pt,
every node/.style={draw,circle,inner sep=1pt},
level distance=.5cm,
level 1/.style={sibling distance=10mm},
level 2/.style={sibling distance=10mm}]

\node  (root) {$5^2$}
  child {node {$6^2$}}
  child {node {$2^1$}
      child{node{$3^2$}}
      child {node {$1^2$}}
      };
\end{tikzpicture}
\end{equation*}

A \textit{right strict binary search tree} is a labelled rooted binary tree where the label of each
node is greater than or equal to the label of every node in its left subtree, and strictly
less than every node in its right subtree.
The associated insertion algorithm is as follows:
\begin{algo}\label{algo:sylv}\cite[\S~3.3]{HNT05}
Input: A right strict binary search tree $T$ and a symbol $a\in \mathcal{A}$.
If $T$ is empty, create a node and label it $a$. If $T$ is non-empty, examine the label $x$
of the root node: if $a > x$, recursively insert $a$ into the right subtree of the root node;
otherwise recursively insert $a$ into the left subtree of the root note. Output the resulting
tree.
\end{algo}

Let $w_1, \cdots, w_k\in \mathcal{A}$ and $w=w_1\cdots w_k \in \mathcal{A}^*$. Then the combinatorial object ${\rm P}_{\sylv_{\infty}}(w)$ of $w$ is obtained as follows: reading
$w$ from right-to-left, one starts with an empty tree and inserts each symbol in $w$ into a right strict binary search tree according to  Algorithm \ref{algo:sylv}.
For example, ${\rm P}_{\sylv_{\infty}}(3613151265)$ is given as follows:
\begin{equation*}
\begin{tikzpicture}
[grow'=down,line width = 0pt,
every node/.style={draw,circle,inner sep=2pt},
level distance=.5cm,
level 1/.style={sibling distance=20mm},
level 2/.style={sibling distance=10mm},
level 3/.style={sibling distance=10mm},
level 4/.style={sibling distance=10mm}]

\node  (root) {5}
  child {node {6}
    child[missing]
    child {node {6}}
  }
  child {node {2}
    child {node {5}
      child[missing]
      child {node {3}
        child[missing]
        child {node {3}}
      }
    }
    child {node {1}
      child[missing]
      child {node{1}
        child[missing]
        child {node {1}}
      }
    }
  };
\end{tikzpicture}
\end{equation*}

A \textit{left strict binary search tree is} a labelled rooted binary tree where the label of each
node is strictly greater than the label of every node in its left subtree, and less than or
equal to every node in its right subtree.
The associated insertion algorithm is as follows:
\begin{algo}\label{algo:sylv-2}
Input: A left strict binary search tree $T$ and a symbol $a \in \mathcal{A}$.
If $T$ is empty, create a node and label it $a$. If $T$ is non-empty, examine the label $x$
of the root node: if $a < x$, recursively insert $a$ into the left subtree of the root node;
otherwise recursively insert $a$ into the right subtree of the root note. Output the resulting
tree.
\end{algo}

Let $w_1, \cdots, w_k\in \mathcal{A}$ and $w=w_1\cdots w_k \in \mathcal{A}^*$. Then the combinatorial object ${\rm P}_{\sylv^{\sharp}_{\infty}}(w)$ of $w$ is obtained as follows: reading $w$ from left-to-right, one starts with an empty tree and inserts each symbol in $w$ into a left strict binary search tree according to  Algorithm \ref{algo:sylv-2}.
For example, ${\rm P}_{\sylv^{\sharp}_{\infty}}(3613151265)$ is given as follows:
\begin{equation*}
\begin{tikzpicture}
[grow'=down,line width = 0pt,
every node/.style={draw,circle,inner sep=2pt},
level distance=.5cm,
level 1/.style={sibling distance=20mm},
level 2/.style={sibling distance=10mm},
level 3/.style={sibling distance=10mm},
level 4/.style={sibling distance=10mm}]

\node  (root) {3}
  child {node {6}
    child{node{6}}
    child {node {3}
      child {node {5}
        child {node {5}}
        child [missing]
        }
      child [missing]
    }
  }
  child {node {1}
    child {node {1}
      child {node {1}
        child {node {2}}
        child[missing]
        }
      child[missing]
      }
     child[missing]
     };
\end{tikzpicture}
\end{equation*}

Let $w_1, \cdots, w_k\in \mathcal{A}$ and $w=w_1\cdots w_k \in \mathcal{A}^*$. Then the combinatorial object  ${\rm P}_{\baxt_{\infty}}(w)$ of $w$ is obtained by the Algorithms \ref{algo:sylv} and \ref{algo:sylv-2}, that is,
${\rm P}_{{\baxt}_{\infty}}(w)= ({\rm P}_{{\sylv_{\infty}^{\sharp}}}(w), {\rm P}_{{\sylv}_{\infty}}(w))$.

For each $\M \in \{\stal,\taig, \sylv,\sylv^{\sharp}, \baxt\}$, define the relation $\equiv_{\M_{\infty}}$ by
\[
u \equiv_{\M_{\infty}} v \Longleftrightarrow {\rm P}_{\M_{\infty}}(u) = {\rm P}_{\M_{\infty}}(v)
\]
for any $u,v \in \mathcal{A}^*$. In each case, the relation $\equiv_{\M_{\infty}}$ is a congruence on $\mathcal{A}^*$. The stalactic monoid $\stal_{\infty}$ [resp. taiga monoid $\taig_{\infty}$, sylvster monoid $\sylv_{\infty}$, $\sharp$-sylvster monoid $\sylv^{\sharp}_{\infty}$, Baxter monoid $\baxt_{\infty}$] is the factor monoid $ \mathcal{A}^*/_{\equiv_{\M_{\infty}}}$. The rank-$n$ analogue $\stal_{n}$ [resp. $\taig_{n}$, $\sylv_{n}$, $\sylv^{\sharp}_{n}$, $\baxt_{n}$] is the factor monoid $ \mathcal{A}^*_n/_{\equiv_{\M_{\infty}}}$, where the relation $\equiv_{\M_{\infty}}$ is naturally restricted to $\mathcal{A}^*_n\times\mathcal{A}^*_n$ and $\mathcal{A}_n = \{1 < 2 < \cdots < n\}$ is set of the first $n$ natural numbers viewed as a finite ordered alphabet.
It follows from the definition of $\equiv_{\M_{\infty}}$ for any $\M \in \{\stal,\taig, \sylv,\sylv^{\sharp}, \baxt\}$ that each element $[u]_{\equiv_{\M_{\infty}}}$
of the factor monoid $\M_{\infty}$ can be identified with the
combinatorial object ${\rm P}_{\M_{\infty}}(u)$.
In each case, $\M_1$ is a free monogenic monoid $\langle a \rangle=\{1, a, a^2, a^3,\ldots\}$ and thus commutative.
Note that
\[
\M_1 \subset \M_2 \subset \cdots \subset \M_i \subset\M_{i+1}\subset \cdots \subset \M_{\infty}.
\]

The \textit{evaluation} of a word $u\in \mathcal{A}^*$, denoted by $\ev(u)$, is the infinite
tuple of non-negative integers, indexed by $\mathcal{A}$, whose $a$-th element, denoted by $|u|_a$, is the number of times the symbol $a$
appears in $u$; thus this tuple describes the number of each symbol in $\mathcal{A}$ that appears in $u$. It is immediate from the definition of the monoids above that if $u \equiv_{\M_\infty} v$, then $\ev(u) = \ev(v)$, and hence it makes
sense to define the evaluation of an element $p$ of one of these monoids to be the evaluation
of any word representing it. We write $\ev(u) \leqslant \ev(v)$ [resp. $\ev(u) < \ev(v)$] if $|u|_a\leqslant |v|_a$ [resp. $|u|_a < |v|_a$] for each non-negative integer $a$.

A \textit{basis} for an algebra $A$ is a set of identities satisfied by $A$ that axiomatize all identities of $A$. An algebra $A$ is said to be \textit{finitely based} if it has some finite basis. Otherwise, it is said to be \textit{non-finitely based}. The finite basis problem, that is the problem of classifying algebras according to the finite basis property, is one of the most prominent research problems in universal algebra.
Since the first example of {\nfb} finite semigroup was discovered by Perkins~\cite{Perkins69} in the 1960s, the finite basis problem for semigroups has attracted much attention. Now there exist several powerful methods to attack the finite basis problem for finite semigroups (see Volkov~\cite{Vol01} for detail).

In contrast with the finite case, the finite basis problem for infinite
semigroups is less explored. On the one hand, infinite semigroups usually arise in
mathematics as transformation semigroups of an infinite set, or semigroups of relations on an infinite domain, or matrix semigroups over an infinite ring. And all these
semigroups are too big to satisfy any non-trivial identity. On the other
hand, when an infinite semigroup does satisfy non-trivial identities, then deciding if
there is a finite basis remains difficult. Indeed, many of methods designed for finite
semigroups do not apply so that fresh techniques are required.



Since the plactic monoid of infinite rank does not satisfy any non-trivial
identity \cite[Proposition 3.1]{CKK17}, the plactic monoid of infinite rank is finitely based. The plactic monoid of rank $2$  satisfies exactly the same identities as the bicyclic monoid \cite[Remark 4.6]{JK19}, or equivalently
\cite[Theorem 4.1]{DJK18} the monoid of all $2\times 2$ upper triangular tropical matrices. Thus the plactic monoid of rank $2$ is non-finitely based by the result of Chen et al. \cite[Corollary 5.6]{CHLS16}. The plactic monoid of rank $3$ satisfies exactly the same identities as the monoid of all $3\times 3$ upper triangular tropical matrices \cite[Corollary 4.5]{JK19}. Thus the plactic monoid of rank $3$ is non-finitely based by the result of Han et al. \cite{HZL}. The finite basis problems for the plactic monoids of rank greater than or equal to $4$ are still open.
Cain et al. proved that all hypoplactic monoids of rank greater than or equal to $2$ are finitely based and satisfy the same identities \cite{C}. For each $\M \in \{\stal,\taig, \sylv,\sylv^{\sharp}, \baxt\}$, $\M_1$ is a free monogenic monoid and commutative, and so $\M_1$ is finitely based by \cite[Theorem~9]{Perkins69}.  However the finite basis problems for $\M_n$ with $2\leq n \leq \infty$ are still open.

In this paper, we investigate the finite basis problems for all stalactic, taiga, sylvester and Baxter monoids of rank greater than or equal to $2$. It is shown that all stalactic and taiga monoids of rank greater than or equal to $2$ are finitely based and satisfy the same identities, that all sylvester monoids of rank greater than or equal to $2$ are finitely based and satisfy the same identities and that all Baxter monoids of rank greater than or equal to $2$ are finitely based and satisfy the same identities.

This paper is organized as follows. Notation and background information of the paper
are given in Section \ref{sec: prelim}. In Section \ref{sec:3sc}, three sufficient conditions under which  semigroups are
finitely based are given. By applying these sufficient conditions, we solve the finite basis problems for all stalactic, taiga,  sylvester and Baxter monoids of rank greater than or equal to 2 in Section \ref{sec:app}.

\section{Preliminaries} \label{sec: prelim}
Most of the notation and background material of this article are given in this section.
Refer to the monograph of Burris and Sankappanavar~\cite{BS81} for more information.

Let~$\mathcal{X}$ be a countably infinite alphabet.
Elements of $\mathcal{X}$ are called \textit{letters} and  elements of the free monoid $\mathcal{X}^*$ are called \textit{words}.
Let $\bfw \in \mathcal{X}^*, x, y, x_1, x_2, \dots , x_m \in \mathcal{X}$. Then
\begin{itemize}
  \item the \textit{content} of $\bfw$, denoted by $\con(\bfw)$, is the set of letters occurring in $\bfw$;

  \smallskip

  \item $\occ(x, \bfw)$ is the number of occurrences of the letter $x$ in $\bfw$;

  \smallskip

  \item $\overleftarrow{\occ}_y(x, \bfw)$ [resp. $(\overrightarrow{\occ}_y(x, \bfw))$] is the number of occurrences of $x$ before [resp. after] the first [resp. last] occurrence of $y$ in $\bfw$;

  \smallskip

  \item $\bfw$ is said to be \textit{simple} if $\occ(x, \bfw) = 1$ for any $x\in \con(\bfw)$;

  \smallskip

  \item the \textit{initial part} [resp. \textit{final part}] of $\bfw$, denoted by $\ip(\bfw)$ [resp. $\fp(\bfw)$], is the simple word obtained from $\bfw$ by retaining the first [resp. last] occurrence of each letter;

  \smallskip

  \item $\mix(\bfw)$ is the word obtained from $\bfw$ by retaining the first and the last occurrences of each letter;

  \smallskip

  \item $\bfw[x_1, x_2, \dots , x_m]$ denote the word obtained from $\bfw$ by retaining only the occurrences of the letters  $x_1, x_2, \dots , x_m$.
\end{itemize}

A \textit{semigroup identity} is a formal expression $\bfu \approx \bfv$ where $\bfu, \bfv$ are words over the alphabet $\mathcal{X}$. An identity $\bfu\approx \bfv$ is said to be \textit{non-trivial} if $\bfu\neq \bfv$ and \textit{trivial} otherwise.
A semigroup $S$ \textit{satisfies} an identity $\bfu\approx \bfv$ if the equality
$\varphi(\bfu) = \varphi(\bfv)$ holds in $S$ for every possible substitution $\varphi : \mathcal{X}\rightarrow S$.  Denote by $\id(S)$ the set of all non-trivial identities satisfied by $S$.

Clearly any  monoid that satisfies an identity $\bfs \approx \bft$ also satisfies
the identity $\bfs[x_1, x_2, \dots , x_n] \approx \bft[x_1, x_2, \dots , x_n]$ for any $x_1, x_2, \dots , x_n \in\mathcal{X}$,
since assigning the unit element  to a letter $x$ in an identity is effectively
the same as removing all occurrences of $x$.

An \textit{identity system} $\Sigma$ is a collection of non-trivial identities. An identity $\bfu\approx \bfv$ is
said to be \textit{derived} from $\Sigma$ or is a \textit{consequence} of $\Sigma$ if there is a
sequence of words
\[
\bfu = \bfu_1, \bfu_2,\cdots , \bfu_{n-1}, \bfu_n = \bfv
\]
over the alphabet $\mathcal{X}$ such that for every $i = 1, 2, \dots , n-1, \bfu_i = \bfa_i\varphi_i(\bfp_i)\bfb_i, \bfu_{i+1} = \bfa_i\varphi_i(\bfq_i)\bfb_i$ with some words $\bfa_i, \bfb_i \in \mathcal{X}^*$, some endomorphism $\varphi_i:\mathcal{X}^+ \rightarrow \mathcal{X}^+$ and some
identity $\bfp_i\approx \bfq_i \in \Sigma$.

Given an identity system $\Sigma$, we
denote by $\id(\Sigma)$ the set of all consequences of $\Sigma$. An \textit{identity basis} for a semigroup
$S$ is any set $\Sigma \subseteq \id(S)$ such that $\id(\Sigma) = \id(S)$, that is, every identity satisfied
by $S$ can be derived from $\Sigma$. A semigroup $S$ is called \textit{finitely based} if it possesses a
finite identity basis, that is, all identities satisfied by $S$ can be derived from a finite
subset of $\id(S)$; otherwise $S$ is called \textit{non-finitely based}. Two semigroups $S_1$ and $S_2$ are called \textit{equationally equivalent} if $\id(S_1) = \id(S_2)$.

For any semigroup $S$, let $S^1$ be the monoid obtained from $S$ by adjoining a unit
element. Denote by $L_2$, $R_2$, $M$ the left-zero semigroup  of order $2$, the right-zero semigroup
of order $2$ and the free monogenic monoid, whose presentations are given as follows:
\begin{align*}
L_2&=\langle a,b~|~a^2=ab=a, b^2=ba=b\rangle,\\
R_2&=\langle a,b~|~a^2=ba=a, b^2=ab=b \rangle,\\
M&=\langle a \rangle=\{1, a, a^2, a^3, \ldots \}.
\end{align*}

The following results are well-known.

\begin{lemma}\label{L21}
Let $\bfu \approx  \bfv$ be any non-trivial identity. Then
\begin{enumerate}[\rm(i)]
  \item $L^1_2$ satisfies $\bfu \approx \bfv$ if and only if $\ip(\bfu) = \ip(\bfv)$;
  \item $R^1_2$ satisfies $\bfu \approx \bfv$ if and only if $\fp(\bfu) = \fp(\bfv)$;
  \item $M$ satisfies $\bfu \approx \bfv$ if and only if $\occ(x,\bfu)=\occ(x,\bfv)$ for any $x \in \mathcal{X}$.
\end{enumerate}
\end{lemma}


\section{Three sufficient conditions for a semigroup to be finitely based}\label{sec:3sc}%

In this section, we give three sufficient conditions under which a
semigroup is finitely based. The next three theorems are the main results in this section.

\begin{theorem}\label{thm:sc}
Suppose that a semigroup $S$  satisfies the identity $xyx \approx yx^2$ and  for any identity $\bfu\approx \bfv$ satisfied by $S$,
\begin{enumerate}[\rm(i)]
\item $\occ(x,\bfu)=\occ(x,\bfv)$ for any $x \in \mathcal{X}$;
\item $\fp(\bfu)=\fp(\bfv)$.
\end{enumerate}
Then the identity $xyx \approx yx^2$ is an identity basis for $S$, and so $S$ is finitely based.
\end{theorem}

\begin{proof}
It suffices to show that any identity $\bfu \approx \bfv$ satisfied by $S$  can be derived from
$xyx \approx yx^2$.
For any $x\in \con(\bfu)$, if $\occ(x, \bfu)\geq 2$, then the identity $xyx\approx yx^2$ can be used to gather any non-last $x$ in $\bfu$ with the last $x$ in $\bfu$, that is $\bfu$ can be written into the form
\begin{align*}\label{cf}
\bfu=x_1^{e_1}x_2^{e_2}\cdots x_m^{e_m}
\end{align*}
where $\fp(\bfu)=x_1x_2\cdots x_m$. By the same argument, $\bfv$ can be written into the form
\[
\bfv=y_1^{f_1}y_2^{f_2}\cdots y_n^{f_n}
\]
where $\fp(\bfv)=y_1y_2\cdots y_n$.
It follows from (ii) that $m=n, x_i=y_i$ for $i=1,2,\dots,m$. And $e_i=f_i$ for $i=1,2,\dots,m$ can be obtained from (i). Therefore $\bfu=\bfv$. Consequently,
every identity satisfied by $S$ is a consequence of the identity $xyx \approx yx^2$, and so the identity $xyx \approx yx^2$ is an identity basis for $S$.
\end{proof}

\begin{theorem}\label{thm:sc1}
Suppose that a semigroup $S$ satisfies the identity $xysxty \approx yxsxty$ and for any identity $\bfu\approx \bfv$ satisfied by $S$,
\begin{enumerate}[\rm(i)]
\item $\occ(x,\bfu)=\occ(x,\bfv)$ for any $x \in \mathcal{X}$;
\item $\overrightarrow{\occ}_y(x, \bfu)=\overrightarrow{\occ}_y(x, \bfv)$ for any $x,y \in \mathcal{X}$;
\item $\fp(\bfu)=\fp(\bfv)$.
\end{enumerate}
Then the identity $xysxty \approx yxsxty$ is an identity basis for $S$, and so $S$ is finitely based.
\end{theorem}

\begin{proof}
It suffices to show that any identity $\bfu \approx \bfv$ satisfied by $S$ can be derived from
$xysxty \approx yxsxty$. Since $S$ satisfies the identity $xysxty \approx yxsxty$, it follows that $\bfu$ can be written into the form
\[
\bfu_1x_1^{e_1}\bfu_2x_2^{e_2}\cdots\bfu_mx_m^{e_m}
\]
where $\fp(\bfu)=x_1\cdots x_m$, $e_i \geq 1$ and $\con(\bfu_i) \subseteq \{x_i, \ldots, x_m\}$ for $1\leq i \leq m$. For each $1\leq i\leq m$, the letters in $\bfu_i$ are not the last occurrences and so can
be moved in any manner by using the identity $xysxty \approx yxsxty$. In particular, any occurrence
of $x_i$ in $\bfu_i$ can be moved to the right and combined with $x_i^{e_i}$ that immediately follows $\bfu_i$. Therefore, we may further assume that $\bfu_i=x_{i+1}^{g_{i+1}}\cdots x_m^{g_m}$  with $g_{i+1}, \ldots, g_m \geq 0$ for $1\leq i\leq m-1$ and $\bfu_m=\emptyset$. By the same argument, $\bfv$ can be written into the form
\[
\bfv=\bfv_1y_1^{f_1}\bfv_2y_2^{f_2}\cdots\bfv_ny_n^{f_n}
\]
where $\fp(\bfv)=y_1y_2\cdots y_n$, $f_i \geq 1$ for $1\leq i \leq n$, $\bfv_i=y_{i+1}^{h_{i+1}}\cdots x_n^{h_n}$ with $h_{i+1}, \ldots, h_n \geq 0$ for $1\leq i\leq n-1$ and $\bfv_n=\emptyset$.
It follows from (iii) that $m=n$ and $x_i=y_i$ for $1\leq i \leq m$. It follows from (i) that $e_1=f_1$. For $1< i \leq m$, since $x_i\not\in \con(\bfu_i\bfv_i)$, it follows that $e_i=\overrightarrow{\occ}_{x_{i-1}}(x_i,\bfu)$ and $f_i=\overrightarrow{\occ}_{x_{i-1}}(x_i,\bfv)$. Hence by (ii) $e_i=f_i$  for  $1< i \leq m$. Therefore $e_i=f_i$  for  $1 \leq i \leq m$.

Clearly, $\bfu_m=\bfv_m=\emptyset$. To show that $\bfu_i=\bfv_i$ for $i=1, 2,\dots, m-1$, it suffices to show that $\occ(x_j,\bfu_i)=\occ(x_j,\bfv_i)$ for $i+1\leq j\leq m$. Since
\begin{align*}
\occ(x_j,\bfu_1)&=\occ(x_j, \bfu)-\overrightarrow{\occ}_{x_1}(x_j, \bfu)\\
&=\occ(x_j, \bfv)-\overrightarrow{\occ}_{x_1}(x_j, \bfv)\\
&=\occ(x_j,\bfv_1)
\end{align*}
by (i) and (ii), it follows that $\bfu_1=\bfv_1$. Since
\begin{align*}
\occ(x_j,\bfu_i)&=\overrightarrow{\occ}_{x_{i-1}}(x_j, \bfu)-\overrightarrow{\occ}_{x_i}(x_j, \bfu)\\
&=\overrightarrow{\occ}_{x_{i-1}}(x_j, \bfv)-\overrightarrow{\occ}_{x_i}(x_j, \bfv)\\
&=\occ(x_j,\bfv_i)
\end{align*}
by (ii), it follows that $\bfu_i=\bfv_i$ for $1<i\leq m-1$.  Therefore $\bfu=\bfv$. Consequently,
every identity satisfied by $S$ is a consequence of the identity $xysxty \approx yxsxty$, and so the identity $xysxty \approx yxsxty$ is an identity basis for $S$.
\end{proof}

\begin{theorem}\label{thm:sc2}
Suppose that a semigroup $S$  satisfies the identities
\begin{subequations}\label{id:baxt}
\begin{gather}
ysxt\,xy\,hxky \approx ysxt\,yx\,hxky, \label{id:a}\\
xsyt\,xy\,hxky \approx xsyt\,yx\,hxky,\label{id:b}
\end{gather}
\end{subequations}
and for any identity $\bfu\approx \bfv$ satisfied by $S$,
\begin{enumerate}[\rm(i)]
\item $\occ(x,\bfu)=\occ(x,\bfv)$ for any $x \in \mathcal{X}$;
\item $\overleftarrow{\occ}_y(x, \bfu)=\overleftarrow{\occ}_y(x, \bfv), \overrightarrow{\occ}_y(x, \bfu)=\overrightarrow{\occ}_y(x, \bfv)$ for any $x,y \in \mathcal{X}$;
\item $\ip(\bfu)=\ip(\bfv), \fp(\bfu)=\fp(\bfv)$.
\end{enumerate}
Then the identities \eqref{id:baxt}
constitute an identity basis for $S$, and so $S$ is finitely based.
\end{theorem}
\begin{proof}
It suffices to show that any identity $\bfu \approx \bfv$ satisfied by $S$ can be derived from \eqref{id:baxt}.
Suppose that $\con(\bfu)=\{x_1,x_2,\dots,x_r\}$ and  $\mix(\bfu)=a_1 a_2 \cdots a_{m+1}$ with $a_i\in \con(\bfu)$ for $i=1,2,\ldots,m+1$. Clearly, $\bfu$ can be written into the form
\[
\bfu = a_1\bfu_1a_2\bfu_2\cdots a_m\bfu_ma_{m+1}
\]
where $\bfu_1,\dots, \bfu_m \in \mathcal{X}^*$. Since each occurrence of letters in $\bfu_i$
is neither its first occurrence nor its last occurrence in $\bfu$,  the letters in $\bfu_i$ can be
permutated within $\bfu_i$ by the identities \eqref{id:baxt} in any manner. Therefore, we may assume that $\bfu_i = x^{i_{j_1}}_1 x^{i_{j_2}}_2\cdots x^{i_{j_r}}_r$ with some non-negative integers $i_{j_1} , i_{j_2} ,\dots, i_{j_r}$. Suppose that $\con(\bfv)=\{y_1,y_2,\dots,y_s\}$ and  $\mix(\bfv)=\{b_1, b_2,\dots, b_{n+1}\}$ with $b_i\in \con(\bfv)$ for $i=1,2,\ldots,n+1$. By the same argument, $\bfv$ can be written into the form
\[
\bfv = b_1\bfv_1b_2\bfv_2\cdots b_n\bfv_nb_{n+1}
\]
where $\bfv_i = y^{i_{k_1}}_1 y^{i_{k_2}}_2\cdots y^{i_{k_s}}_s$ with some non-negative integers $i_{k_1} , i_{k_2} ,\dots, i_{k_s}$.
In the following, we will show that if $\bfu \approx \bfv$ satisfies the conditions (i)--(iii), then $\bfu=\bfv$.

It follows from (i) that either $\occ(x, \bfu)= \occ(x, \bfv)=1$ or $\occ(x, \bfu)=\occ(x, \bfv)\geq 2$. Hence $m=n$.

Next we show that $\mix(\bfu) = \mix(\bfv)$.
Clearly, $a_1 = b_1$ by (iii). Proceeding by induction, suppose that $a_1 \cdots a_{k-1} = b_1 \cdots b_{k-1}$. Assume that $a_k = x$ and $b_k = y$.
If both $a_k$ and $b_k$ are the first occurrences
of $x$ in $\bfu$ and $y$ in $\bfv$ respectively, then it follows from $\ip(\bfu)=\ip(\bfv)$ that $a_k = b_k$; if both $a_k$ and $b_k$ are the last occurrences of $x$ in $\bfu$ and $y$ in $\bfv$ respectively, then it follows from $\fp(\bfu)=\fp(\bfv)$ that $a_k = b_k$. Otherwise, by symmetry, we may assume that $a_k$ is the first occurrence of $x$ in $\bfu$ and $b_k$ is the last occurrence
of $y$ in $\bfv$. Then by the above arguments the result $a_k = b_k$ still holds when either $\occ(x,\bfu)=1$ or $\occ(y,\bfv)=1$. Therefore, we may assume that $\occ(x,\bfu), \occ(y,\bfv)\geq 2$.

Suppose that $a_k \neq b_k$. Then $x \not\in \con(a_1\bfu_1\cdots a_{k-1}\bfu_{k-1})$ by $a_k$ being the first occurrence of $x$ in $\bfu$ and $y \not\in \con(\bfv_kb_{k+1} \cdots \bfv_nb_{n+1})$ by $b_k$ being the last occurrence of $y$ in $\bfv$. Hence it follows from $\occ(y,\bfv)\geq 2$ that $\occ(y, b_1\cdots b_{k-1})=1$, so that $\occ(y, a_1\cdots a_{k-1})=\occ(y, a_{k+1}\cdots a_{n+1})=1$. Clearly, $x \not\in\con(a_1\cdots a_{k-1})=\con(b_1\cdots b_{k-1})$. Hence
$\occ(y, \bfv)=\overleftarrow{\occ}_x(y,\bfv)$. Now by (i) and (ii),
\[
\occ(y, \bfu) = \occ(y, \bfv)=\overleftarrow{\occ}_x(y,\bfv) = \overleftarrow{\occ}_x(y,\bfu).
\]
But $\occ(y, \bfu)= \overleftarrow{\occ}_x(y,\bfu)$ is impossible since $\occ(y, a_{k+1}\cdots a_{n+1})=1$, hence $a_k = b_k$.
Therefore $a_i = b_i$ for all $i = 1,...,n + 1$ by induction, and so $\mix(\bfu) = \mix(\bfv)$.

Finally, we show that $\bfu_k =\bfv_k$ for each $k=1, \ldots, n$. By the forms of $\bfu$ and $\bfv$, it suffices to show
that $\occ(z, \bfu_k) = \occ(z, \bfv_k)$ for any $z \in \con(\bfu_k\bfv_k)$ and $k = 1, 2,\ldots, n$. Let
$\occ(z, \bfu_k) = s$ and $\occ(z, \bfv_k) = t$. There are two cases.

{\bf Case~1.} $a_k = a_{k+1} = x$. Then $a_k$ and $a_{k+1}$ are the first and the last occurrences of $x$ in both $\bfu$ and $\bfv$.
If $z = x$, then by (i),
\[
2 + s = \occ(x, \bfu) = \occ(x, \bfv)=2+ t.
\]
Hence $s = t$.
If $z \neq x$, then by (i),
\[
\overleftarrow{\occ}_x(z, \bfu) + s + \overrightarrow{\occ}_x(z, \bfu) = \occ(z, \bfu) = \occ(z, \bfv) =\overleftarrow{\occ}_x(z, \bfv) + t + \overrightarrow{\occ}_x(z, \bfv).
\]
Thus $s = t$ follows from (ii).

{\bf Case~2.} $a_k = x \neq y = a_{k+1}$. By symmetry, there are three subcases.

{\bf 2.1.} $a_k$ and $a_{k+1}$ are the first occurrences of $x$ and $y$ respectively in both $\bfu$ and $\bfv$. Clearly, $z \neq y$.
If $z = x$, then by (ii),
\[
1 + s = \overleftarrow{\occ}_y(x, \bfu) =\overleftarrow{\occ}_y(x, \bfv)=1+ t.
\]
Hence $s = t$.
If $z \neq x$, then by (ii),
\[
\overleftarrow{\occ}_x(z, \bfu) + s = \overleftarrow{\occ}_y(z, \bfu) = \overleftarrow{\occ}_y(z, \bfv) = \overleftarrow{\occ}_x(z, \bfv) + t.
\]
Thus $s = t$ follows from (ii).

{\bf 2.2.} $a_k$ is the first occurrence of $x$ and $a_{k+1}$ is the last occurrence
of $y$ in both $\bfu$ and $\bfv$.
If $z = x$, then by (i),
\[
1 + s + \overrightarrow{\occ}_y(x, \bfu) = \occ(x, \bfu) = \occ(x, \bfv)=1+ t + \overrightarrow{\occ}_y(x, \bfv).
\]
Thus $s = t$ follows from (ii).
If $z = y$, then by (i),
\[
\overleftarrow{\occ}_x(y, \bfu) + s +1= \occ(y, \bfu) = \occ(y, \bfv) = \overleftarrow{\occ}_x(y, \bfv) + t + 1.
\]
Thus $s = t$ follows from (ii).
If $z \neq x, y$, then by (i),
\[
\overleftarrow{\occ}_x(z, \bfu) + s + \overrightarrow{\occ}_y(z, \bfu) = \occ(z, \bfu) = \occ(z, \bfv) = \overleftarrow{\occ}_x(z, \bfv) + t + \overrightarrow{\occ}_y(z, \bfv).
\]
Thus $s = t$ follows from (ii).

{\bf 2.3.} $a_k$ is the last occurrence of $x$ and $a_{k+1}$ is the first occurrence
of $y$ in both $\bfu$ and $\bfv$. Clearly, $z \neq x, y$. Then by (i),
\[
\overrightarrow{\occ}_x(z, \bfu) + \overleftarrow{\occ}_y(z, \bfu) - s = \occ(z, \bfu) = \occ(z, \bfv) = \overrightarrow{\occ}_x(z, \bfv) + \overleftarrow{\occ}_y(z, \bfv) - t.
\]
Thus $s = t$ follows from (ii).

Hence $\bfu_k = \bfv_k$ for $k = 1, 2,\cdots,n$. Therefore $\bfu = \bfv$. Consequently,
every identity satisfied by $S$ is a consequence of the identities in \eqref{id:baxt}, and so the identities \eqref{id:baxt} constitute an identity basis for $S$.
\end{proof}

\section{Finite basis problems for stalactic, taiga,  sylvester and Baxter monoids}\label{sec:app}%

In this section, by applying the sufficient conditions given in Section \ref{sec:3sc}, we solve the finite basis problems for all stalactic, taiga, sylvester and Baxter monoids of rank greater than or equal to $2$.

\subsection{Finite basis problems for stalactic and taiga monoids}

\begin{lemma}\cite[Propositions 15 and 16]{CM16}\label{lem::stal id}
Both the stalatic monoid $\stal_{\infty}$ and the taiga monoid $\taig_{\infty}$ satisfy the identity $xyx \approx yx^2$.
\end{lemma}

\begin{theorem}
 The identity $xyx\approx yx^2$ is a finite identity basis for the monoids $\stal_n$ and $\taig_n$ whenever $2\leq n\leq \infty$. Therefore all stalactic and taiga monoids of rank greater than or equal to $2$ are equationally equivalent.
\end{theorem}

\begin{proof}
Clearly, we only need to show that each of the monoids $\stal_n$ and $\taig_n$ for any $2\leq n\leq \infty$ can be defined by the identity $xyx\approx yx^2$.
First we show that each of the monoids  $\taig_n$ for any $2\leq n \leq \infty$ can be defined by the identity $xyx \approx yx^2$.
Note that
\[
\taig_1 \subset \taig_2 \subset \cdots \subset \taig_n \subset \cdots \subset \taig_{\infty}.
\]
By Theorem \ref{thm:sc}, it suffices to show that $\taig_{\infty}$ satisfies the identity  $xyx\approx yx^2$ and $\taig_2$ satisfies the conditions (i) and (ii) in Theorem \ref{thm:sc}. Clearly, $\taig_{\infty}$ satisfies the identity  $xyx\approx yx^2$ by Lemma~\ref{lem::stal id}.

Let $\bfu\approx \bfv$ be any identity satisfied by the monoid $\taig_2$. Since $\taig_1$ is a free monogenic monoid, it follows from Lemma~\ref{L21} (iii) that $\occ(x,\bfu)=\occ(x,\bfv)$ for any $x \in \mathcal{X}$, and so the condition (i) holds in $\taig_2$. Suppose that $\fp(\bfu)\neq \fp(\bfv)$. Then there exist some $x,y$ such that $\taig_2$ satisfies $\bfa yx^s=\bfu[x,y]\approx \bfv[x,y]=\bfb xy^t$ for some $s,t \geq 1$ and $\bfa, \bfb \in \{x,y\}^*$. Let $\varphi$ be a substitution such that $x\mapsto 2, y\mapsto 1$. Then $\varphi(\bfu[x,y])$ ends with $2$ and $\varphi(\bfv[x,y])$ ends with $1$. Since the rightmost symbol in a word $w$ determines the root node of $\mathrm{P}_{\taig_2}(w)$, it follows that $\varphi(\bfu[x,y]) \ne \varphi(\bfv[x,y])$.  This implies that $\taig_2$ does not satisfy $\bfu[x,y]\approx \bfv[x,y]$, a contradiction. Hence $\fp(\bfu)= \fp(\bfv)$, and so the condition (ii) holds.

Next we show that each of the monoids $\stal_n$ for any $2\leq n\leq \infty$ can be defined by the identity $xyx \approx yx^2$.
Note that
\[
\stal_1 \subset \stal_2 \subset \cdots \subset \stal_n \subset \cdots \subset \stal_{\infty}.
\]
By Theorem \ref{thm:sc}, it suffices to show that $\stal_{\infty}$ satisfies the identity  $xyx\approx yx^2$ and $\stal_2$ satisfies the conditions (i) and (ii) in Theorem \ref{thm:sc}. Clearly, $\stal_{\infty}$ satisfies the identity  $xyx\approx yx^2$ by Lemma~\ref{lem::stal id}. It is routine to show that $\stal_2$ is isomorphic to $\taig_2$. Hence $\stal_2$ satisfies the conditions (i) and (ii) in Theorem \ref{thm:sc} by the above arguments.

Consequently, each of monoids $\stal_n, \taig_n$ for any $2\leq n\leq \infty$ can be defined by the identity $xyx \approx yx^2$, and so all of them are equationally equivalent.
\end{proof}

\subsection{Finite basis problem for sylvester monoid}

\begin{lemma}\label{lem:property of sylv'id}
Let $\bfu\approx \bfv$ be any identity satisfied by the monoid $\sylv_2$. Then the monoid $\sylv_2$ satisfies the conditions (i)--(iii) in Theorem \ref{thm:sc1}.
\end{lemma}

\begin{proof}
Let $\bfu\approx \bfv$ be any identity satisfied by $\sylv_2$. Since $\sylv_1$ is a free monogenic monoid, it follows from Lemma~\ref{L21} (iii) that $\occ(x,\bfu)=\occ(x,\bfv)$ for any $x \in \mathcal{X}$, and so the condition (i) holds in $\sylv_2$.
Suppose that $\overrightarrow{\occ}_y(x, \bfu)\neq \overrightarrow{\occ}_y(x, \bfv)$ for some  $x,y \in \mathcal{X}$. Then $\sylv_2$ satisfies $\bfa yx^s=\bfu[x,y]\approx \bfv[x,y]=\bfb yx^t$ for some $s, t \geq 1$, $s\ne t$ and $\bfa, \bfb \in \{x,y\}^*$. Without loss of generality, we may assume that $s < t$. Let $\phi$ be a substitution such that $x\mapsto 2, y\mapsto 1$. Using the Algorithm \ref{algo:sylv}, one sees that
\begin{align*}
\mathrm{P}_{\sylv_2}(\phi(\bfu[x,y]))=\parbox[c]{2.5cm}
{\begin{tikzpicture}
[line width = 0pt,
empty/.style = {circle, draw, inner sep=2pt}]
\node [empty,label = right:$1$-th] (A) at (2,2) {2};
\node [empty, label = right:$2$-th] (B) at (1.6,1.2) {2};
\node [empty, label = right:$s$-th] (C) at (1.2,.4) {2};
\node [empty, label = right:$(s+1)$-th] (D) at (.8,-.4) {1};
\node [rectangle,draw, minimum size = 0.4cm] (E) at (.4,-1.2) {};
\node [rectangle,draw, minimum size = 0.4cm] (F) at (1.2,-1.2) {};
\draw (A) -- (B);
\draw[dashed] (B) -- (C);
\draw (C) -- (D);
\draw (D) -- (E);
\draw (D) -- (F);
\end{tikzpicture}}
\text{and}\quad
\mathrm{P}_{\sylv_2}(\phi(\bfv[x,y]))=\parbox[c]{3cm}
{\begin{tikzpicture}
[line width = 0pt,
empty/.style = {circle, draw, inner sep=2pt}]
\node [empty,label = right:$1$-th] (A) at (2,2) {2};
\node [empty, label = right:$2$-th] (B) at (1.6,1.2) {2};
\node [empty, label = right:$s$-th] (C) at (1.2,.4) {2};
\node [empty, label = right:$t$-th] (D) at (.8,-.4) {2};
\node [empty, label = right:$(t+1)$-th] (E) at (.4,-1.2) {1};
\node [rectangle,draw, minimum size = 0.4cm] (F) at (0,-2) {};
\node [rectangle,draw, minimum size = 0.4cm] (G) at (.8,-2) {};
\draw (A) -- (B);
\draw[dashed] (B) -- (C);
\draw[dashed] (C) -- (D);
\draw (D) -- (E);
\draw (E) -- (F);
\draw (E) -- (G);
\end{tikzpicture}}
\end{align*}
Then $\varphi(\bfu[x,y]) \ne \varphi(\bfv[x,y])$.  This implies that $\sylv_2$ does not satisfy $\bfu[x,y]\approx \bfv[x,y]$, a contradiction. Hence $\overrightarrow{\occ}_y(x, \bfu)=\overrightarrow{\occ}_y(x, \bfv)$ for any $x,y \in \mathcal{X}$, and so the condition (ii) holds.

Suppose that $\fp(\bfu)\neq \fp(\bfv)$. Then there exists letters $x,y$ such that $\sylv_2$ satisfies $\bfa yx^s=\bfu[x,y]\approx \bfv[x,y]=\bfb xy^t$ for some $s, t\geq 1$ and $\bfa, \bfb \in \{x,y\}^*$. Let $\varphi$ be a substitution such that $x\mapsto 2, y\mapsto 1$. Then $\varphi(\bfu[x,y])$ ends with $2$ and $\varphi(\bfv[x,y])$ ends with $1$.  Since the rightmost symbol in a word $w$ determines the root node of $\mathrm{P}_{\sylv_2}(w)$, it follows that $\varphi(\bfu[x,y]) \ne \varphi(\bfv[x,y])$.  This implies that $\sylv_2$ does not satisfy $\bfu[x,y]\approx \bfv[x,y]$, a contradiction. Hence $\fp(\bfu)= \fp(\bfv)$, and so the condition (iii) holds.
\end{proof}

\begin{lemma}\label{lem:pr=qr}
Let $p, q, r \in \sylv_{\infty}$  such that $\ev(p) = \ev(q)\leqslant \ev(r)$. Then $pr = qr$.
\end{lemma}

\begin{proof}
In \cite[Lemma 19]{CM16}, it is shown that if $p, q, r \in \sylv_{\infty}$ such that $\ev(p) = \ev(q)= \ev(r)$, then $pr = qr$. In fact, by the proof of \cite[Lemma 19]{CM16}, it is easy to see that the result still holds when $\ev(p) =\ev(q)<\ev(r)$. This is because every symbol $d$ that from $p$ or $q$ is inserted into a particular previously empty subtree of ${\rm P}_{\sylv_{\infty}}(r)$, dependent only on the value of the symbol $d$ (and not on its position in $p$ or $q$), and that unequal symbols are inserted into different subtrees. Since $\ev(p) = \ev(q)$, the same number of symbols $d$ are inserted for each such symbol $d$. Hence if $p, q, r \in \sylv_{\infty}$ such that $\ev(p) = \ev(q)\leqslant \ev(r)$, then $pr = qr$ still holds.
\end{proof}

\begin{theorem}\label{cor:sylv's id}
The sylvester monoid $\sylv_{\infty}$ satisfies the identity $xysxty \approx yxsxty$.
\end{theorem}

\begin{proof}
Let $\varphi :\mathcal{X}\rightarrow \sylv_{\infty}$ be any substitution. Then it is obvious that  $\ev(\varphi(xy))=\ev(\varphi(yx))\leqslant \ev(\varphi(sxty))$. Hence it follows from Lemma \ref{lem:pr=qr} that $\varphi(xysxty)=\varphi(yxsxty)$. Therefore the sylvester monoid $\sylv_{\infty}$ satisfies the identity $xysxty \approx yxsxty$.
\end{proof}

\begin{theorem}
The identity $xysxty \approx yxsxty$ is a finite identity basis for the monoids $\sylv_n$ whenever $2\leq n\leq \infty$. Therefore all sylvster monoids of rank greater than or equal to $2$ are equationally equivalent.
\end{theorem}

\begin{proof}
Clearly, we only need to show that each of the monoids $\sylv_n$ for any $2\leq n\leq \infty$ can be defined by the identity $xysxty \approx yxsxty$.  Note that
\[
\sylv_1 \subset \sylv_2 \subset \cdots \subset \sylv_n \subset \cdots \subset \sylv_{\infty}.
\]
By Theorem \ref{thm:sc1}, it suffices to show that $\sylv_{\infty}$ satisfies the identity  $xysxty \approx yxsxty$ and $\sylv_2$ satisfies the conditions (i)--(iii) in Theorem \ref{thm:sc1}. Therefore, the results hold directly follows from Theorem~\ref{cor:sylv's id} and Lemma \ref{lem:property of sylv'id}.
\end{proof}

Symmetrically, we have
\begin{theorem}
The identity $ytxsyx \approx ytxsxy$ is a finite identity basis for the monoids $\sylv_n^\sharp$ whenever $2\leq n\leq \infty$. Therefore all $\sharp$-sylvster monoids of rank greater than or equal to $2$ are equationally equivalent.
\end{theorem}

\subsection{Finite basis problem for Baxter monoid}

\begin{lemma}\label{lem:spr=sqr}
Let $p, q, r, s \in \baxt_{\infty}$ such that $\ev(p) = \ev(q) \leqslant \ev(r), \ev(s)$. Then
$spr = sqr$.
\end{lemma}

\begin{proof}
Since $\ev(p) = \ev(q) \leqslant \ev(r)$, it follows from Lemma \ref{lem:pr=qr} that ${\rm P}_{\sylv_{\infty}}(pr)={\rm P}_{\sylv_{\infty}}(qr)$. Thus ${\rm P}_{\sylv_{\infty}}(spr)={\rm P}_{\sylv_{\infty}}(sqr)$.  Since $\ev(p) = \ev(q) \leqslant \ev(s)$, it follows from the dual of Lemma \ref{lem:pr=qr} that ${\rm P}_{\sylv^{\sharp}_{\infty}}(sp)={\rm P}_{\sylv^{\sharp}_{\infty}}(sq)$. Thus ${\rm P}_{\sylv^{\sharp}_{\infty}}(spr)={\rm P}_{\sylv^{\sharp}_{\infty}}(sqr)$.  Therefore ${\rm P}_{\baxt_{\infty}}(spr)={\rm P}_{\baxt_{\infty}}(sqr)$, so that $spr = sqr$.
\end{proof}

\begin{theorem}\label{cor:baxt's id}
The Baxter monoid $\baxt_{\infty}$ satisfies the identities \eqref{id:baxt}.
\end{theorem}

\begin{proof}
Let $\varphi :\mathcal{X}\rightarrow \baxt_{\infty}$ be any substitution.
It is obvious that  $\ev(\varphi(xy))=\ev(\varphi(yx))\leqslant \ev(\varphi(ysxt)), \ev(\varphi(hxky))$. By Lemma \ref{lem:spr=sqr}, we have  $\varphi(ysxtxyhxky) =  \varphi(ysxtyxhxky)$. Therefore the Baxter monoid $\baxt_{\infty}$ satisfies the identity \eqref{id:a}. A similar argument can show that the Baxter monoid $\baxt_{\infty}$ satisfies the identity \eqref{id:b}.
\end{proof}

\begin{theorem}\label{thm:baxt'id}
The identities \eqref{id:baxt} constitute a finite identity basis for the monoids $\baxt_n$ whenever $2\leq n\leq \infty$. Therefore all Baxter monoids of rank greater than or equal to $2$ are equationally equivalent.
\end{theorem}
\begin{proof}
Clearly, we only need to show that each of the monoids $\baxt_n$ for any $2\leq n\leq \infty$ can be defined by the identities \eqref{id:baxt}.
Note that
\[
\baxt_1 \subset \baxt_2 \subset \cdots \subset \baxt_n \subset \cdots \subset \baxt_{\infty}.
\]
By Theorem \ref{thm:sc2}, it suffices to show that $\baxt_{\infty}$ satisfies the identities \eqref{id:baxt} and $\baxt_2$ satisfies the conditions (i)--(iii) in Theorem \ref{thm:sc2}.

Clearly, the Baxter monoid $\baxt_{\infty}$ satisfies the identities \eqref{id:baxt} by Theorem \ref{cor:baxt's id}.
Since  both $\sylv_2$ and $\sylv_2^{\sharp}$ are homomorphic images of $\baxt_2$ by the definition of Baxter monoid, it follows from Lemma \ref{lem:property of sylv'id} and its dual that $\baxt_2$ satisfies the conditions (i)--(iii) in Theorem \ref{thm:sc2}. Consequently, each of monoids $\baxt_n$ for any $2\leq n\leq \infty$ can be defined by the identities \eqref{id:baxt}, and so all of them are equationally equivalent.
\end{proof}

\end{document}